\documentclass[11pt]{amsart}
\usepackage{float}
\usepackage{tikz}
\usepackage{amsmath, amsthm, color, fullpage}
\usepackage{graphicx, amsaddr}
\usepackage{enumerate}
\usepackage{epsfig}
\usepackage{epstopdf}
\usepackage{tikz}
\usepackage{array,amsmath, enumerate,  url, psfrag}
\usepackage{amssymb, amsaddr, fullpage}
\usepackage{graphicx,subfigure}
\usepackage{color}
\newtheorem{theorem}{Theorem}[section]

\newtheorem{lemma}[theorem]{Lemma}

\title{List-recoloring of two classes of planar graphs}
\author{Chenran Pan, Weifan Wang, Runrun Liu}

\address{
\small School of Mathematics, Zhejiang Normal University, Jinhua, Zhejiang 321004, China.
}

\thanks{The research of the second author was supported by NSFC, China (No.12031018). The research of the third author was supported by NSFC (No.\,12101563), ZJNSFC (No.\,LQ22A010011). }

\email{2213882766@qq.com,rliu1206@zjnu.edu.cn,wwf@zjnu.cn}

\begin{document}
\maketitle
\begin{abstract}
For a graph $G$ with a list assignment $L$ and two $L$-colorings $\alpha$ and $\beta$, an $L$-recoloring sequence from $\alpha$ to $\beta$ is a sequence of proper $L$-colorings where consecutive colorings differ at exactly one vertex. We prove the existence of such a recoloring sequence in which every vertex is recolored at most a constant number of times under two conditions: (i) $G$ is planar, contains no $3$-cycles or intersecting $4$-cycles, and $L$ is a $6$-assignment; or (ii) the maximum average degree of $G$ satisfies $\mathrm{mad}(G) < \frac{5}{2}$ and $L$ is a $4$-assignment. These results strengthen two theorems previously established by Cranston.
\end{abstract}
 \baselineskip=16pt

\section{Introduction}
Graph recoloring is a fundamental topic in graph theory and reconfiguration, focusing on transforming one proper coloring of a graph into another through a sequence of elementary steps. Given a graph $G$, a \emph{proper $k$-coloring} is an assignment of colors from $\{1, 2, \dots, k\}$ to the vertices of $G$ such that no two adjacent vertices share the same color. A \emph{recoloring step} changes the color of a single vertex while maintaining a proper coloring. For a $k$-colorable graph $G$, the $k$-recoloring graph $\mathcal{R}_k(G)$ is a graph whose vertices correspond to the $k$-colorings of $G$, and two vertices of $\mathcal{R}_k(G)$ are adjacent if their corresponding $k$-colorings of $G$ differ at exactly one vertex. For two given proper $k$-colorings $\alpha$ and $\beta$, a central question is whether $\alpha$ can be transformed into $\beta$ via a sequence of such steps, and if so, how many steps are required.

This problem can be generalized to \emph{list coloring}, where each vertex $v$ is assigned a list $L(v)$ of available colors. An \emph{$L$-coloring} is a proper coloring in which each vertex receives a color from its list. The \emph{$L$-recoloring graph} $\mathcal{R}_L(G)$ has as vertices all $L$-colorings of $G$, with two colorings adjacent if they differ at exactly one vertex. A key objective is to determine whether $\mathcal{R}_L(G)$ is connected and to bound its diameter¡ªthe maximum number of steps needed to reconfigure between any two $L$-colorings.

Early work in this area was motivated by applications in statistical physics and optimization. Cereceda~\cite{C07} conjectured that for any $d$-degenerate graph $G$ and any $k \geq d+2$, the $k$-recoloring graph is connected and has diameter $O(n^2)$, where $n = |V(G)|$. This remains open even for $d = 2$. It is well known that triangle-free planar graphs are $3$-degenerate. Bousquet and Heinrich~\cite{BH22} proved that if $G$ is a planar bipartite graph, then $\operatorname{diam}(\mathcal{R}_5(G))=O(n^2)$, confirming Cereceda's conjecture for the class of planar bipartite graphs. Recently, Cranston and Mahmoud~\cite{CM24} strengthened this result by proving that if $G$ is a planar graph without $3$-cycles and $5$-cycles, then $\operatorname{diam}(\mathcal{R}_5(G))=O(n^2)$.

Significant progress has been made for sparse graphs~\cite{BP16,C22}, planar graphs~\cite{BBFHMP23,CGP25,CM24,DF20,DF21}, and $H$-free graphs~\cite{BC24,BCM24,BK25+,LMMSW24}. Bartier et al.~\cite{BBFHMP23} proved that if $G$ is a planar graph of girth at least $6$, then $\operatorname{diam}(\mathcal{R}_5(G))=O(n)$. Dvo\v{r}'{a}k and Feghali~\cite{DF20,DF21} proved that $\operatorname{diam}(\mathcal{R}_{10}(G))\le 8n$ for every planar graph $G$, and conjectured that $\operatorname{diam}(\mathcal{R}_L(G))=O(n)$ for $10$-list-coloring of planar graphs, $7$-list-colorings of triangle-free planar graphs, and $6$-list-colorings of planar graphs of girth at least $5$. These conjectures were confirmed by Cranston~\cite{C22,C25}, who also extended some of the results to graphs with bounded maximum average degree. Cambie et al.~\cite{CBC24} conjectured that $\operatorname{diam}(\mathcal{R}_L(G))\le n(G)+\mu(G)$, where $\mu(G)$ is the size of a maximum matching of $G$, and proved several results towards this conjecture.

More recently, attention has turned to optimizing the number of recoloring times at a single vertex. A recoloring sequence is \emph{$k$-good} if each vertex is recolored at most $k$ times. Cranston~\cite{C22} showed the following two results.

\begin{theorem}(\cite{C22})
Let $G$ be a planar graph without $3$-cycles and $4$-cycles. Fix a $6$-assignment $L$ for $G$ and $L$-colorings $\alpha$ and $\beta$. There exists an $L$-coloring sequence that transforms $\alpha$ into $\beta$ such that each vertex is recolored at most $12$ times.
\end{theorem}

\begin{theorem}(\cite{C22})
Let $G$ be a graph with $\operatorname{mad}(G)<\frac{22}{9}$. Fix a $4$-assignment $L$ for $G$ and $L$-colorings $\alpha$ and $\beta$. There exists an $L$-coloring sequence that transforms $\alpha$ into $\beta$ such that each vertex is recolored at most $12$ times.
\end{theorem}

In this paper, we make progress on the above two results and prove the following two main theorems, using a unified method based on discharging and the key lemma in graph recoloring.

\begin{theorem}\label{main1}
Let $G$ be a planar graph without $3$-cycles and intersecting $4$-cycles. Fix a $6$-assignment $L$ for $G$ and $L$-colorings $\alpha$ and $\beta$. There exists an $L$-coloring sequence that transforms $\alpha$ into $\beta$ such that each vertex is recolored at most $48$ times.
\end{theorem}

\begin{theorem}\label{main2}
Let $G$ be a graph with $\operatorname{mad}(G)<\frac{5}{2}$. Fix a $4$-assignment $L$ for $G$ and $L$-colorings $\alpha$ and $\beta$. There exists an $L$-coloring sequence that transforms $\alpha$ into $\beta$ such that each vertex is recolored at most $18$ times. In particular, this holds for every planar graph with girth at least $10$.
\end{theorem}

\section{Preliminaries }

An \emph{$L$-recoloring sequence} is a sequence of $L$-recoloring steps, starting from some specified $L$-coloring, such that each subsequent coloring is a proper $L$-coloring. An $L$-recoloring sequence is \emph{$k$-good} if each vertex is recolored at most $k$ times. An $L$-coloring $\alpha$ of $G$ restricted to a subgraph $G'$ is denoted $\alpha_{\upharpoonright G'}$. We will often extend an $L$-recoloring sequence $\sigma'$ that transforms $\alpha_{\upharpoonright G'}$ to $\beta_{\upharpoonright G'}$ to an $L$-recoloring sequence $\sigma$ that transforms $\alpha$ to $\beta$. To prove Theorems~\ref{main1} and~\ref{main2}, we will use the following simple but powerful Key Lemma. This lemma has been used implicitly in many papers and first appeared explicitly in~\cite{BBFHMP23}; later, Cranston~\cite{C22} phrased it in the slightly more general language of list coloring.

\begin{lemma}[Key Lemma~\cite{C22}]
Fix a graph $G$, a list assignment $L$ for $G$, and $L$-colorings $\alpha$ and $\beta$. Fix a vertex $v$ with $|L(v)|\ge d(v)+2$ and let $G'=G-v$. Fix an $L$-recoloring sequence $\sigma'$ for $G'$ transforming $\alpha_{\upharpoonright G'}$ to $\beta_{\upharpoonright G'}$. If $\sigma'$ recolors $N_G(v)$ a total of $t$ times, then we can extend $\sigma'$ to an $L$-recoloring sequence for $G$ transforming $\alpha$ into $\beta$ such that $v$ is recolored at most $\lceil t/(|L(v)|-d_G(v)-1)\rceil+1$ times.
\end{lemma}

\section{Proof of Theorem~\ref{main1} }
In this section, we give the proof of Theorem~\ref{main1}. For a vertex $v$, we call a {\em $k$-vertex}, or {\em $k^+$-vertex}, or {\em $k^-$-vertex} if $d(v)=k$, or $d(v)\ge k$, or $d(v)\le k$, respectively.  Similarly, we define {\em $k$-face}, {\em $k^+$-face}, and {\em a $k^-$-face}.
A $k$-face $f=[v_1v_2\cdots v_k]$ is an $(x_1,x_2\ldots,x_k)$-{\em face} if $d(v_i)=x_i$ for each $i \in[k]$. If a $k$-vertex $v$ is adjacent to a vertex $u$, then we call $v$ a $k$-neighbor of $u$. We assume that $G$ is a minimum counterexample of Theorem~\ref{main1}. The following Lemma~\ref{3.1} to ~\ref{3.3} are involved in the proof of Theorem 2 in \cite{C22}.
\begin{lemma}\label{3.1}(Theorem 2 in ~\cite{C22})
$G$ contains no $2^-$-vertex.
\end{lemma}
\begin{lemma}\label{3.2}(Theorem 2 in ~\cite{C22})
No $3$-vertex is adjacent to two $3$-vertices in $G$.
\end{lemma}

\begin{lemma}\label{3.3}(Theorem 2 in ~\cite{C22})
No $4$-vertex is adjacent to four $3$-vertices in $G$.
\end{lemma}

\begin{lemma}\label{3.4}
Let $f=[v_1v_2v_3v_4v_5]$ be a $(3,4,3,3,4)$-face. Then each of $v_2$ and $v_4$ cannot be adjacent to a $3$-vertex $u$ with a $3$-neighbor $w$ such that $u,w\notin V(f)$.
\end{lemma}

\begin{proof}
Suppose otherwise, by symmetry assume that $v_2$ is adjacent to a $3$-vertex $u$ with a $3$-neighbor $w$ such that $u,w\notin V(f)$.
Let $G'$:=$G-\{v_1,v_2,v_3,v_4,v_5,x,y\}$. By minimality, $G'$ has a $48$-good L-recoloring sequence $\sigma'$ that transforms $\alpha_{\upharpoonright G'}$ to $\beta_{\upharpoonright G'}$. By the Key lemma, we can recolor $v_2$ at most $\lceil\frac{48}{4}\rceil+1=13$ times and $v_5$ and $y$ at most $\lceil\frac{2\times48}{3}\rceil+1=33$ times. Then $v_1$ and $x$ can be recolored at most $\lceil\frac{33+48+13}{3}\rceil+1=48$ times by again using the Key lemma.  Then we can recolor $v_4$ at most $\lceil\frac{33+48}{3}\rceil+1=28$ times and $v_3$ at most $\lceil\frac{28+48+13}{2}\rceil+1=46$ times by the Key lemma. Therefore, we get a $48$-good L-recoloring sequence $\sigma$ that transforms $\alpha_{\upharpoonright G}$ to $\beta_{\upharpoonright G}$ of $G$, a contradiction.
\end{proof}

\begin{lemma}\label{3.5}
Let $f=[v_1v_2v_3v_4v_5]$ be a $(3,4,3,3,4)$-face, then $f$ cannot be adjacent to a $(3,4,3,4)$-face.
\end{lemma}
\begin{proof}
Suppose otherwise that $f$ is adjacent to a $(3,4,3,4)$-face. By symmetry either $f_1=[v_1v_2uw]$ or $f_2=[v_3v_2uw]$ is a $(3,4,3,4)$-face.  Let $G'$:=$G-\{v_1,v_2,v_3,v_4,v_5,u,w\}$. By minimality, $G'$ has a $48$-good $L$-recoloring sequence $\sigma'$ that transforms $\alpha_{\upharpoonright G'}$ to $\beta_{\upharpoonright G'}$. By the Key lemma, $v_2$ can be recolored at most $\lceil\frac{48}{4}\rceil+1=13$ times and $v_5$ and $w$ can be recolored at most $\lceil\frac{2\times48}{3}\rceil+1=33$ times. Next we recolors each of $u,v_1$ at most $\lceil\frac{13+48+33}{2}\rceil+1=48$ times by the Key lemma. Now we recolor $v_4$ at most $\lceil\frac{33+48}{3}\rceil+1=28$ times by the Key lemma. Then we can extend $\sigma'$ to a $48$-good $L$-recoloring sequence $\sigma$ of $G$ that transforms $\alpha_{\upharpoonright G}$ to $\beta_{\upharpoonright G}$ by recoloring $v_3$ at most $\lceil\frac{28+13+48}{2}\rceil+1=46$ times.
\end{proof}

We apply a discharging argument on $G$ to finish the proof of Theorem~\ref{main1}.
For $v\in V(G)$,$f\in F(G)$, let $\mu(v)=2d(v)-6$,$\mu(f)=d(f)-6$ be the initial charge. By Euler's formula,
$\sum_{v\in V(G)}(2d(v)-6)+\sum_{f\in F(G)}(d(f)-6)=-12$.
We will redistribute the charges of $G$ according to a set of rules, to make every $x\in V(G)\cup F(G)$ satisfies $\mu^*(x)\ge 0$. Se we have $-12=\sum_{x\in V(G)\cup F(G)}\mu(x)=\sum_{x\in V(G)\cup F(G)}\mu^*(x)\ge 0$, which is a contradiction.

For $x,y\in V(G)\cup F(G)$, let $\tau(x\to y)$ be the charge that $x$ gives $y$ according to the following rules. A $5$-face $f=[v_1v_2v_3v_4v_5]$ is called special if it is a $(3,4,3,3,4)$-face such that $v_1v_2$ is on a $(3,4,4,3)$-face. We call $v_2$ a poor $4$-vertex to $f$ and $v_5$ a rich $4$-vertex to $f$. The discharging rules are as follows:
\begin{itemize}
\item[(R1)] Each $5^+$-vertex gives $\frac{4}{3}$ to each incident $4$-face and $\frac{2}{3}$ to incident $5$-face.

\item[(R2)] Each special $5$-face gets $\frac{2}{3}$ from each incident rich $4$-vertex and $\frac{1}{3}$ from each incident  poor $4$-vertex.

\item[(R3)] After (R1)(R2)  each $4$- or $5$-face with negative charge gets needed charge evenly from its incident $4$-vertices.
\end{itemize}

\begin{lemma}\label{charge}
Let $v$ be a $4$-vertex in $G$. Then $v$ gives $1$ to each incident $(3,4,3,4)$-face or $(3,4,3,4)$-face, $\frac{2}{3}$ to each incident $(3,4,4,4)$-face, $(3,4,3,5^+)$-face, $(3,3,4,5^+)$-face or special $5$-face $f$ when $v$ is rich to $f$, and at most $\frac{1}{2}$ to each of other incident faces.
\end{lemma}

\begin{proof}
Obviously $v$ gives $1$ to each incident $(3,4,3,4)$-face or $(3,4,3,4)$-face, $\frac{2}{3}$ to each incident $(3,4,4,4)$-face, $(3,4,3,5^+)$-face or special $5$-face $f$ when $v$ is rich to $f$. Let $f$ be an incident face of $v$ which is not the faces mentioned above. Since $G$ contains no $3$-cycles, $d(f)\ge4$. If $d(f)=4$, then $f$ must be a $(3,4,4^+,5^+)$-face or  $(4,4^+,4^+,4^+)$-face. By (R1)(R3) $v$ gives at most $\frac{1}{3}<\frac{1}{2}$ to $f$ in the former case and at most $\frac{1}{2}$ to $f$ in the latter case.  If $d(f)=5$, then $f$ is incident to at least two $4^+$-vertices by Lemma~\ref{3.2}. If $f$ is incident to a $5^+$-vertex, then by (R1) $f$ gets $\frac{2}{3}$ from the $5^+$-vertex. So $v$ gives at most $\frac{1}{3}<\frac{1}{2}$ to $f$.  If $f$ is incident to no $5^+$-vertices, then $f$ must be incident to at least three $4$-vertices or $f$ is a $(3,4,3,3,4)$-face.In the former case, $v$ gives at most $\frac{1}{3}<\frac{1}{2}$ to $f$ by(R1) (R3). In the latter case, if $f$ is special, then $v$ is poor to $f$. So by (R2) $v$ gives $\frac{1}{3}<\frac{1}{2}$ to $f$.  If $f$ is not special, then by (R3) $v$ gives $\frac{1}{2}$ to $f$. If $d(f)\ge6$, then $v$ gives nothing to $f$ since $6^+$-faces are involved in the discharging procedure.
 \end{proof}

We shall show that each $x\in V(G)\cup F(G)$ has final charge $\mu^*(x)\ge 0$. Let $f$ be a $k$-face in $G$. Since $G$ contains no $3$-cycles, $k\ge4$. If $k\in\{4,5\}$, then $\mu^*(f)\ge0$ by (R3). If $k\ge6$, then $\mu^*(f)\ge d(f)-6\ge0$ since $6^+$-faces are not involved in the discharging procedure.

Let $v$ be a $k$-vertex in $G$. By Lemma~\ref{3.1} $k\ge3$. If $k=3$, then $\mu^*(v)=\mu(v)=2\times3-6=0$ since $3$-vertices are not involved in the discharging procedure. If $k\ge5$, then $v$ gives $\frac{4}{3}$ to at most one $4$-face and gives at most $\frac{2}{3}$ to each of other incident faces by R(1). So we have $\mu^*(v)\geq 2k-6-\frac{4}{3}- \frac{2}{3}(k-1) \geq 0$. Now we are left to consider the case that $k=4$. Let $N_G(v)=\{v_1,v_2,v_3,v_4\}$ in clockwise order and let $f_i$ be the incident  face of $v$ containing $v_i,v,v_{i+1}$ for each $i\in[4]$ (index module $4$).  By Lemma~\ref{charge} $v$ gives at most $\frac{1}{2}$ to each incident face which implies that $\mu^*(v)\ge2\times4-6-\frac{1}{2}\times4=0$ unless one of the following four cases holds.

{\bf Case 1:}\  $v$ is incident to a $(3,4,3,4^+)$-face, say $f_1$. If $f_1$ is a $(3,4,3,4)$-face, then by (R3) $\tau(v\to f_1)=1$. By Lemma~\ref{3.5}, $f_2$ or $f_4$ cannot be a $(3,4,3,3,4)$-face.  By Lemma~\ref{3.3} $f_3$ can not be a $(3,4,3,3,4)$-face. So we have $\tau(v\to f_i)\leq \frac{1}{3}$ for $i\in\{2,3,4\}$. Thus, $\mu^*(v)\geq 2\times4-6-1-\frac{1}{3}\times3=0$. If $f_1$ is a $(3,4,3,5^+)$-face, then  by (R1)(R3) $\tau(v\to f_1)=\frac{2}{3}$. By Lemma~\ref{3.3} at least one of $\{v_3,v_4\}$ is a $4^+$-vertex, by symmetry say $v_3$, then $\tau(v\to f_2)+\tau(v\to f_3)\leq \frac{1}{3}+\frac{1}{3}$. By (R3) $\tau(v\to f_4)\le\frac{2}{3}$ since $f_4$ is a $5^+$-face. So  $\mu^*(v)\geq 2-\frac{2}{3}\times2-\frac{1}{3}\times2=0$.

 {\bf Case 2:}\  $v$ is incident to a $(3,4,4^+,3)$-face, say $f_1$.
 Let $f_1$ be a $(3,4,4^+,3)$-face with $d(v_1)=3$. If $f_1$ is a $(3,4,4,3)$-face, then by (R3) $\tau(v\to f_1)=1$. If $f_4$ is a $(3,4,3,3,4)$-face, then $v$ is a poor $4$-vertex to $f_4$. so $v$ gives $\frac{1}{3}$ to $f_4$ by (R2). In other cases, $v$ gives at most $\frac{1}{3}$ to $f_4$. By Lemma~\ref{3.4}, $f_3$ cannot be a $(3,4,3,3,4)$-face. Since $d(v)=d(v_2)=4$, $f_2$ also cannot be a $(3,4,3,3,4)$-face. So $\tau(v\to f_2)+\tau(v\to f_3)\leq \frac{1}{3}+\frac{1}{3}=\frac{2}{3}$ by R(1)R(3). The final charge $\mu^*(v)\geq 2\times4-6-1-\frac{1}{3}\times3\geq0$. If $f_1$ is a $(3,4,5^+,3)$-face, then $v$ gives $\frac{2}{3}$ to $f_1$, at most $\frac{1}{3}$ to $f_2$ and at most $\frac{2}{3}$ to $f_3$. So $\mu^*(v)\geq 2-\frac{2}{3}\times2-\frac{1}{3}\times2=0$.

{\bf Case 3:}\  $v$ is incident to a $(3,4,4,4)$-face, say $f_1$.
 Let $f_1$ be a $(3,4,4,4)$-face with $d(v_1)=3$, so by (R3)  $\tau(v\to f_1)=\frac{2}{3}$. Since $v,v_2$ are $4$-vertices, then $f_2$ is incident to at least three $4^+$-vertices yb Lemma~\ref{3.2}. So $\tau(v\to f_2) \leq \frac{1}{3}$ by (R1)(R3). Since $v$ is not a rich vertex on a special $5$-face, $v$ gives at most $\frac{1}{2}$ to each of $f_3$ and $f_4$. So $\mu^*(v)\geq 2-\frac{2}{3}-\frac{1}{3}-\frac{1}{2}\times2=0$.

{\bf Case 4:} $v$ is incident to a special $5$-face, say $f_1=[v_1vv_2uw]$ and $v$ is rich to $f_1$, $w$ is a poor $4$-vertex by symmetry. Then $\tau(v\to f_1)= \frac{2}{3}$ by (R2) and $f_4$ must be a $5^+$-face since $G$ contains no intersecting $4$-faces.  By Lemma~\ref{3.4} each of $f_2$ and $f_4$ cannot be a $(3,4,3,3,4)$-face. So $\tau(v\to f_4)\leq \frac{1}{3}$ by (R1)(R3).

 {\bf Case 4.1:}\  $f_2, f_3$ are $5^+$-faces.
By Lemma~\ref{3.3}, at least one of $v_3,v_4$ is a $4^+$-vertex. Let $v_i$ be a $4^+$-vertex with $i\in\{3,4\}$. Then each of $f_{i-1}$ and $f_i$ is incident to at least three $4^+$-vertices by Lemma~\ref{3.2}. So $\tau(v\to f_i)= \frac{1}{3}$ and $v$ gives each other incident face at most $\frac{2}{3}$. So  $\mu^*(v)\geq2\times4-6-\frac{1}{3}\times2-\frac{2}{3}\times2=0$.

{\bf Case 4.2:}\  one of $f_2, f_3$ is a $4$-face.
If $f_2$ is a $4$-face, then let $f_2=[vv_3xv_2]$. Since $d(u)=d(v_2)=3$, $d(x)\geq4$ by Lemma~\ref{3.2}.  By Lemma~\ref{3.3} $v_3$ or $v_4$ is a $4^+$-vertex. So by Lemma~\ref{3.2} $f_3$ is a $5^+$-face with at least three $4^+$-vertices,  so $\tau(v\to f_3)\leq \frac{1}{3}$ by (R1)(R3). If $d(v_3)\geq4$, then $f_2$ is a $(3,4,4^+,4^+)$-face, so $\tau(v\to f_2)\leq \frac{2}{3}$ by (R1)(R3). If $d(v_3)=3$, then $x$ is a $5^+$-vertex by Lemma~\ref{3.5}. So we have $\tau(v\to f_2)\leq \frac{2}{3}$ by R(1)R(3). Thus, $\mu^*(v)\geq2\times4-6-\frac{2}{3}\times2-\frac{1}{3}\times2=0$. If $f_3$ is a $4$-face, then let $f_3=[vv_3yv_4]$. By Lemma~\ref{3.3} and \ref{3.4} at least two of $\{v_3,y,v_4\}$ are $4^+$-vertices. So $\tau(v\to f_3)\leq \frac{2}{3}$ by (R1)(R3). Recall that $f_2$ cannot be a $(3,4,3,3,4)$-face. So $\tau(v\to f_2)\leq \frac{1}{3}$ by (R1)(R3). So $\mu^*(v)\geq2\times4-6-\frac{2}{3}\times2-\frac{1}{3}\times2=0$.

\section{Proof of Theorem~\ref{main2} }

In this section, we give the proof of Theorem~\ref{main2}.
We first introduce some notations used in this part. A {\em thread} in $G$ is a path with all internal vertices of degree $2$. A {\em $k$-thread} is a thread with $k$ internal vertices. A $3_{a,b,c}$-vertex is a $3$-vertex that is the endpoint of a maximal $a$-thread, a maximal $b$-thread, a maximal $c$-thread, all distinct. A $4_{a,b,c,d}$-vertex is defined similarly. A weak neighbor of a $3^+$-vertex $v$ is a $3^+$-vertex $u$ such that $v$ and $u$ are endpoints of a common 1-thread. In this case, we call $v$ is weak adjacent to $u$ and vice versa. A $2$-vertex $v$ is nearby a $3^+$-vertex $w$ if $v$ is an interior vertex of a thread with $w$ as one endpoint. Call a $3_{1,1,0}$-vertex $v$ bad if $v$ is weak adjacent to a $3_{1,1,1}$-vertex.  We assume that $G$ is a minimum counterexample of Theorem~\ref{main1}. We will use discharging method to show that $G$ has average degree at least $\frac{5}{2}$, contradicting the hypothesis. The following Lemma~\ref{minimumdegree} to ~\ref{3111} are from \cite{C22}. Note that the value of $k$ from Lemma~\ref{2-thread} and ~\ref{3-thread} is $14$ in~\cite{C22}, actually the proof works equally well for all integer $k\ge3$.

\begin{lemma}\label{minimumdegree} (Lemma 2 in ~\cite{C22})
The graph $G$ is connected and $\delta(G)\ge 2$.
\end{lemma}

\begin{lemma}\label{2-thread} (Lemma 3 in ~\cite{C22})
Let $v_1v_2v_3v_4$ be a $2$-thread in some subgraph $H$ of $G$. Let $H'=H-\{v_2,v_3\}$. For $k\ge3$, let ${\sigma}'$ be a $k$-good L-coloring sequence for $H'$ that transforms $\alpha|H'$ to $\beta|H'$. If ${\sigma'}$ recolors $v_4$ at most $s$ times, and $s\le k-3$, then $H$ has a $k$-good L-recoloring sequence that recolors $v_3$ at most $s+3$ times and transforms $\alpha$ to $\beta$.
\end{lemma}

\begin{lemma}\label{3-thread}(Lemma 4 in ~\cite{C22})
Let $v_1v_2v_3v_4v_5$ be a $3$-thread in some subgraph $H$ of $G$(not necessarily proper). Let $H'=H-\{v_2,v_3,v_4\}$. For $k\ge3$, let ${\sigma'}$ be a $k$-good $L$-recoloring sequence for $H'$ that transforms $\alpha_{\upharpoonright H'}$ to $\beta_{\upharpoonright H'}$. We can extend ${\sigma'}$ to a $k$-good L-recoloring sequence for $H$ that transforms $\alpha|H$ to $\beta|H$ and recolors $v_3$ at most $4$ times. In particularly, $G$ has no $3$-threads.
\end{lemma}

\begin{lemma}\label{nearby2}(Lemma 6 in ~\cite{C22})
$G$ does not contain any $3$-vertex with $4$ or more nearby $2$-vertices, and $G$ does not contain any $4$-vertex with $6$ or more nearby $2$-vertices.
\end{lemma}

\begin{lemma}\label{3210}(Lemma 7 in ~\cite{C22})
No $3_{2,1,0}$-vertex is adjacent to a $3_{1,1,0}$-vertex or a $3_{2,0,0}$-vertex or a $3_{2,1,0}$-vertex.
\end{lemma}

\begin{lemma}\label{3111}(Lemma 8 in ~\cite{C22})
Each $3_{1,1,1}$-vertex has no weak $3$-neighbor that is a $3_{2,1,0}$-vertex and no weak $3$-neighbor that is a $3_{1,1,1}$-vertex.
\end{lemma}

%For convenience, for a $k$-vertex $v$ with $k\in\{3,4\}$, let $v$ be a $k_{a_1,a_2,\cdots,a_k}$-vertex with $N(v)=\{v_1,v_2,\cdots,v_k\}$ such that $v_i$ is on a $a_i$-thread for each $i\in[k]$. Without loss of generality, assume that $a_1\ge a_2\ge\cdots\ge a_k$. By Lemma~\ref{3-thread} $a_1\le2$. Denote $n_2(v)=a_1+a_2+\cdots+a_k$. Let $N(v_i)=\{v,u_i\}$ if $d(v_i)=2$.

\begin{lemma}\label{3000}
Each $3$-vertex $v$ in $G$ is adjacent to at most one $3_{2,1,0}$-vertex.
\end{lemma}

\begin{proof}
Let $N(v)=\{v_1,v_2,v_3\}$. Suppose otherwise that $v$ is adjacent to at least two $3_{2,1,0}$-vertices, by symmetry say $v_1,v_2$. Form $G'$ from $G$ by deleting the interior vertices of the two $2$-threads incident to $v_1$ and $v_2$. Form $G''$ from $G'$ by deleting $v_i$ and the interior vertex of a $1$-thread incident to $v$ for each $i\in[2]$. Let $G'''$=$G''$-$v$. By minimality, $G'''$ has a $18$-good L-recoloring sequence that transforms $\alpha_{\upharpoonright G'''}$ to $\beta_{\upharpoonright G'''}$. Since $v_3$ is recolored at most $18$ times, by the Key lemma, we can extend $\sigma'''$ to a $18$-good L-recoloring sequence $\sigma''$ for $G''$ that recolors $v$ at most $\lceil{\frac{18}{2}}\rceil+1=10$ times  and transforms $\alpha_{\upharpoonright G''}$ to $\beta_{\upharpoonright G''}$. By Lemma~\ref{2-thread}, we can extend $\sigma''$ to a $18$-good $L$-recoloring sequence $\sigma'$ for $G'$ that transforms $\alpha_{\upharpoonright G'}$ to $\beta_{\upharpoonright G'}$ and recolors each of $v_1$ and $v_2$ at most $10+3=13$ times. By applying Lemma~\ref{2-thread} for each $2$-thread incident to $v_1$ and $v_2$, we can extend $\sigma'$ to a $18$-good L-recoloring sequence $\sigma$ for $G$ that transforms $\alpha$ to $\beta$.
\end{proof}

\begin{lemma}\label{3100}
Let $v$ be a $3_{1,0,0}$-vertex in $G$ with a $3_{2,1,0}$-neighbor. Then $v$ cannot be  weak adjacent to a $3_{1,1,1}$-vertex or adjacent to a bad $3_{1,1,0}$-vertex.
\end{lemma}

\begin{proof}
Let $N(v)=\{v_1,v_2,v_3\}$ with $d(v_1)=2$. Let $u_1$ be the neighbor of $v_1$ distinct from $v$ and $v_2$ be the $3_{2,1,0}$-neighbor of $v$. First assume to the contrary that $v$ is weak adjacent to a $3_{1,1,1}$-vertex.
Form $G'$ from $G$ by deleting the interior vertices of the $2$-thread incident to $v_2$. Form $G''$ from $G'$ by deleting $v_1,v_2$ and the interior vertex of $1$-thread incident to $v_2$, Form $G'''$ from $G''$ by deleting $v,u_1$ and the interior vertices of two $1$-threads incident to $u_1$. By minimality, $G'''$ has a $18$-good L-recoloring sequence that transforms $\alpha_{\upharpoonright G'''}$ to $\beta_{\upharpoonright G'''}$. We can extend $\sigma'''$ to $18$-good L-recoloring sequence $\sigma''$ for $G''$ that recolors $v$ at most $\lceil{\frac{18}{2}}\rceil+1=10$ times by the Key lemma and recolors $u_1$ at most $4$ times by Lemma~\ref{3-thread} and transforms $\alpha_{\upharpoonright G''}$ to $\beta_{\upharpoonright G''}$. We can extend $\sigma''$ to a $18$-good L-recoloring sequence $\sigma'$ for $G'$ that recolors $v_1$ at most $\lceil\frac{10+4}{4-2-1}\rceil+1=15$ times by the Key lemma and recolors $v_2$ at most $10+3=13$ times by Lemma~\ref{2-thread} and transforms $\alpha_{\upharpoonright G'}$ to $\beta_{\upharpoonright G'}$. By applying Lemma~\ref{2-thread} for the $2$-thread incident to $v_2$, we can extend to a $18$-good L-recoloring sequence for $G$ that transforms $\alpha$ to $\beta$.

Now suppose otherwise that $v_3$ is a bad $3_{1,1,0}$-vertex. Let $v_3'$ be the weak $3_{1,1,1}$-neighbor of $v_3$. Let $N(v_3)=\{v,u,w\}$ and $N(w)=\{v_3,v_3'\}$. Let $G'=G-w$, $G''=G'-\{u,v_3\}$ and $G'''=G''-\{v,v_1\}$. Form $G''''$ from $G'''$ by deleting $v_2$ and the interior vertices of $1$- and $2$-threads incident to $v_2$ and deleting $v_3'$ and the vertices in $N(v_3')-w$. By minimality, $G''''$ has a $18$-good L-recoloring sequence that transforms $\alpha_{\upharpoonright G''''}$ to $\beta_{\upharpoonright G''''}$. We can extend $\sigma''''$ to a $18$-good L-recoloring sequence $\sigma'''$ for $G'''$ that recolors each of $v_2,v_3'$ at most $4$ times by applying Lemma~\ref{3-thread} twice and transforms $\alpha_{\upharpoonright G'''}$ to $\beta_{\upharpoonright G'''}$. We can extend $\sigma'''$ to a $18$-good L-recoloring sequence $\sigma''$ for $G''$ that recolors $v$ at most $4+3=7$ times by Lemma~\ref{2-thread} and transforms $\alpha_{\upharpoonright G''}$ to $\beta_{\upharpoonright G''}$. We can extend $\sigma''$ to a $18$-good L-recoloring sequence $\sigma'$ for $G'$ that recolors $v_3$ at most $7+3=10$ times by Lemma~\ref{2-thread} and transforms $\alpha_{\upharpoonright G'}$ to $\beta_{\upharpoonright G'}$. By applying the Key Lemma to $v_2$, We can extend to a $18$-good $L$-recoloring sequence for $G$ that recolors $v_2$ at most $\lceil\frac{10+4}{4-2-1}\rceil+1=15$ times and transforms $\alpha$ to $\beta$.
\end{proof}

\begin{lemma}\label{3200}
A bad $3_{1,1,0}$-vertex $v$ cannot be adjacent to  a $3_{2,0,0}$-vertex or $3_{1,1,0}$-vertex.
\end{lemma}

\begin{proof}
Let $N(v)=\{v_1,v_2,v_3\}$ with $d(v_1)=d(v_2)=2$. Let $v_1'$ be the weak $3$-neighbor of $v$ and $N(v_1')=\{v_1,u,w\}$.

Suppose otherwise that $v_3$ is a $3_{2,0,0}$-vertex with a $2$-neighbor $v_3'$. Let $G'=G-v_1$. Form $G''$ from $G'$ by deleting $v,v_2$ and the interior vertices of the $2$-thread incident to $v_3$. Form $G'''$ from $G''$ by deleting $v_3,v_1',u,w$. By minimality, $G'''$ has a $18$-good L-recoloring sequence that transforms $\alpha_{\upharpoonright G'''}$ to $\beta_{\upharpoonright G'''}$. We can extend $\sigma'''$ to a $18$-good L-recoloring sequence $\sigma''$ for $G''$ that recolors $v_1'$ at most $4$ times by Lemma~\ref{3-thread} and $v_3$ at most $\frac{18}{2}+1=10$ times by the Key Lemma and transforms $\alpha_{\upharpoonright G''}$ to $\beta_{\upharpoonright G''}$. We can extend $\sigma''$ to a $18$-good L-recoloring sequence $\sigma'$ for $G'$ that recolors each of $v,v_3'$ at most $10+3=13$ times by Lemma~\ref{2-thread} and transforms $\alpha_{\upharpoonright G'}$ to $\beta_{\upharpoonright G'}$. Now by the Key lemma, we can extend $\sigma'$ to a $18$-good L-recoloring sequence $\sigma$ for $G$ that recolors $v_1$ at most $\lceil\frac{13+4}{4-2-1}\rceil+1=18$ times and transforms $\alpha$ to $\beta$.

Suppose otherwise that $v_3$ is a $3_{1,1,0}$-vertex. Let $G'=G-v_1$ and $G''=G'-\{v,v_2\}$. Form $G'''$ from $G''$ by deleting $v_3,v_1',u,w$ and the interior vertices of two $1$-threads incident to $v_3$. By minimality, $G'''$ has a $18$-good L-recoloring sequence that transforms $\alpha_{\upharpoonright G'''}$ to $\beta_{\upharpoonright G'''}$. We can extend $\sigma'''$ to a $18$-good L-recoloring sequence $\sigma''$ for $G''$ that recolors each of $v_3$ and $v_1'$ at most $4$ times by applying Lemma~\ref{3-thread} twice and transforms $\alpha_{\upharpoonright G''}$ to $\beta_{\upharpoonright G''}$. We can extend $\sigma''$ to a $18$-good L-recoloring sequence $\sigma'$ for $G'$ that recolors $v$ at most $4+3=7$ times by Lemma~\ref{2-thread} and transforms $\alpha_{\upharpoonright G'}$ to $\beta_{\upharpoonright G'}$. Now by the Key lemma, we can extend $\sigma'$ to a $18$-good L-recoloring sequence $\sigma$ for $G$ that recolors $v_1$ at most $\lceil\frac{7+4}{4-2-1}\rceil+1=12$ times and transforms $\alpha$ to $\beta$.
\end{proof}

\begin{lemma}\label{3110}
Each $3_{1,1,0}$-vertex $v$ is weak adjacent to at most one $3_{1,1,1}$-vertex.
\end{lemma}
\begin{proof}
Let $N(v)=\{v_1,v_2,v_3\}$ with $d(v_1)=d(v_2)=2$. Suppose to the contrary that $v_1',v_2'$ are two $3_{1,1,1}$-vertices weak adjacent to $v$. Let $G'=G-\{v_1,v_2\}$. Form $G''$ from $G'$ by deleting $v,v_1',v_2'$ and the vertices in $N(v_1')\cup N(v_2')-\{v_1,v_2\}$. By minimality, $G''$ has a $18$-good L-recoloring sequence that transforms $\alpha_{\upharpoonright G''}$ to $\beta_{\upharpoonright G''}$. Then we can extend $\sigma''$ to a $18$-good L-recoloring sequence $\sigma'$ for $G'$ that recolors $v$ at most $\frac{18}{2}+1=10$ times by the Key Lemma and recolors each of $v_1',v_2'$ at most 4 times by applying Lemma~\ref{3-thread} twice and transforms $\alpha_{\upharpoonright G'}$ to $\beta_{\upharpoonright G'}$. We extend $\sigma'$ to a $18$-good L-recoloring sequence for $G$ that recolors each of $v_1,v_2$ at most $\lceil\frac{10+4}{4-2-1}\rceil+1=15$ times and transforms $\alpha$ to $\beta$.
\end{proof}

\begin{lemma}\label{42210}
Each $4_{2,2,1,0}$-vertex $v$ is not adjacent to a $3_{2,1,0}$-vertex.
\end{lemma}
\begin{proof}
Let $N(v)=\{v_1,v_2,v_3,v_4\}$ with $d(v_1)=d(v_2)=d(v_3)=2$ and $v_3$ is in a $1$-thread. Suppose otherwise that $v$ is adjacent to a $3_{2,1,0}$-vertex $v_4$. Form $G'$ from $G$ by deleting the interior vertices of the two $2$-threads incident to $v$. Form $G''$ from $G'$ by deleting $v,v_3$. Form $G'''$ from $G''$ by deleting $v_4$ and the interior vertices of the $1$- and $2$-threads incident to $v_4$. By minimality, $G'''$ has a $18$-good L-recoloring sequence that transforms $\alpha_{\upharpoonright G'''}$ to $\beta_{\upharpoonright G'''}$. We can extend $\sigma'''$ to a $18$-good L-recoloring sequence $\sigma''$ for $G''$ that recolors $v_4$ at most $4$ times and transforms $\alpha_{\upharpoonright G''}$ to $\beta_{\upharpoonright G''}$. We can extend $\sigma''$ to a $18$-good L-recoloring sequence $\sigma'$ for $G'$ that recolors $v$ at most $4+3=7$ times by Lemma~\ref{2-thread} and transforms $\alpha_{\upharpoonright G'}$ to $\beta_{\upharpoonright G'}$.  Again by applying Lemma~\ref{2-thread} twice, we can extend $\sigma'$ to a $18$-good L-recoloring sequence $\sigma$ for $G$ that transform $\alpha$ to $\beta$.
\end{proof}

We will finish Theorem~\ref{main1} by way of discharging method.
For each vertex $v\in V(G)$, let its initial charge be $\mu(v)=d(v)$. Then $\Sigma_{v\in V(G)}<\frac{5}{2}|V(C)|$ and we will redistribute the charge among vertices so that each vertex $v$ ends with charge at least $\mu^*(v)\ge{\frac{5}{2}}$, which is a contradiction. We use the following discharging rules.

\begin{itemize}
\item[(R1)] Each $3^+$-vertex gives $\frac{1}{4}$ to each nearby $2$-vertex.

\item[(R2)] Each $3^+$-vertex gives $\frac{1}{4}$ to each adjacent $3_{2,1,0}$-vertex.

\item[(R3)] Every $3^+$-vertex gives $\frac{1}{12}$ to each adjacent bad $3_{1,1,0}$-vertex.

\item[(R4)] Each $3^+$-vertex gives $\frac{1}{12}$ to each weak adjacent  $3_{1,1,1}$-vretex.
\end{itemize}

Let $v$ be a $k$-vertex in $G$. By Lemma~\ref{minimumdegree}, $k\ge2$. If $k=2$, then by Lemma~\ref{3-thread} $v$ is a nearby $2$-vertex of two $3^+$-vertices. By (R1) $v$ gets $\frac{1}{4}$ from each adjacent $3^+$-vertex. So $\mu^*(v)=2+{\frac{1}{4}}\times2=\frac{5}{2}$.
If $k\ge5$, then let $N(v)=\{v_i:1\le i\le5\}$. For each $i\in[5]$, $v_i$ is either on a $2$-thread $vv_iu_iw$ or on a $1$-thread $vv_iw$ or $d(v_i)\ge3$. In the first case, $v$ gives total $\frac{1}{4}\times2=\frac{1}{2}$ to $v_i,u_i$ by (R1). In the second case, $v$ gives at most $\frac{1}{4}+\frac{1}{12}=\frac{1}{3}$ to $v_i,w$ by (R1)(R4). In the last case, $v$ gives at most $\frac{1}{4}$ to $v_i$ by (R2)-(R3). So  $\mu^*(v)\ge d(v)-\max\{\frac{1}{2}, \frac{1}{3},\frac{1}{4}\}d(v)\ge\frac{5}{2}$. So it is sufficient to show that $\mu^*(v)\ge0$ for $k=3,4$.

For convenience, for a $k$-vertex $v$ with $k\in\{3,4\}$, let $v$ be a $k_{a_1,a_2,\cdots,a_k}$-vertex with $N(v)=\{v_1,v_2,\cdots,v_k\}$ such that $v_i$ is on a $a_i$-thread for each $i\in[k]$. Without loss of generality, assume that $a_1\ge a_2\ge\cdots\ge a_k$. By Lemma~\ref{3-thread} $a_1\le2$. Denote $n_2(v)=a_1+a_2+\cdots+a_k$. Let $N(v_i)=\{v,u_i\}$ if $d(v_i)=2$.

{\bf Case 1.} $d(v)=3$. By Lemma~\ref{nearby2} $n_2(v)\le3$.

$\bullet$ $n_2(v)=0$. Then $v$ is a $3_{0,0,0}$-vertex. By Lemma~\ref{3000}, $v$ is adjacent to at most one  $3_{2,1,0}$-vertex, if has, say $v_1$. By (R2)-(R4) $v$ gives at most $\frac{1}{4}$ to $v_1$, at most $\frac{1}{12}$ to each of $v_2,v_3$. So $\mu^*(v)\ge3-\frac{1}{4}-\frac{1}{12}\times2=\frac{31}{12}>\frac{5}{2}$.

$\bullet$ $n_2(v)=1$.  Then $v$ is a $3_{1,0,0}$-vertex. If $v$ is adjacent to a $3_{2,1,0}$-vertex, say $v_2$, then $v_3$ cannot be a $3_{2,1,0}$-vertex by Lemma~\ref{3000}, and by Lemma~\ref{3100} $u_1$ cannot be a $3_{1,1,1}$-vertex and $v_3$ cannot be a  a bad $3_{1,1,0}$-vertex. So $v$ gives nothing to $u_1$ and $v_3$. By (R1)(R2) $v$ gives $\frac{1}{4}$ to each of $v_1$ and $v_2$. So $\mu^*(v)\ge3-\frac{1}{4}\times2=\frac{5}{2}$. Otherwise, $v$ gives $\frac{1}{4}$ to $v_1$ by (R1) and at most $\frac{1}{12}$ to each of $u_1,v_2,v_3$ by (R3)(R4). So $\mu^*(v)\ge3-\frac{1}{4}-\frac{1}{12}\times3=\frac{5}{2}$.

$\bullet$ $n_2(v)=2$. Then $v$ is a $3_{2,0,0}$-vertex or $3_{1,1,0}$-vertex.
If $v$ is a $3_{2,0,0}$-vertex, then by Lemma~\ref{3210} $v$ cannot be adjacent to a $3_{2,1,0}$-vertex and by Lemma~\ref{3200} $v$ cannot be adjacent to a bad $3_{1,1,0}$-vertex. So $v$ gives nothing to each of $v_2$ and $v_3$. By (R1) $v$ gives $\frac{1}{4}$ to each of $v_1,u_1$. So $\mu^*(v)\ge3-{\frac{1}{4}}\times2=\frac{5}{2}$. So we may assume that $v$ is a $3_{1,1,0}$-vertex. By Lemma~\ref{3210} $v_3$ cannot be a $3_{2,1,0}$-vertex and by  Lemma~\ref{3200} $v_3$ cannot be a bad $3_{1,1,0}$-vertex. So $v$ gives nothing to $v_3$. By (R1) $v$ gives $\frac{1}{4}$ to each of $v_1,v_2$. By  Lemma~\ref{3110} at most one of $u_1$ and $u_2$ is a $3_{1,1,1}$-vertex. If none of $u_1, u_2$ is  a $3_{1,1,1}$-vertex, then $\mu^*(v)\ge3-\frac{1}{4}\times2=0$. If one of $u_1, u_2$ is  a $3_{1,1,1}$-vertex, say $u_1$, then $v$ is a bad $3_{1,1,0}$-vertex. So $v$ gives $\frac{1}{12}$ to $u_1$ by (R4) and gets $\frac{1}{12}$ from $v_3$ by (R3). So $\mu^*(v)\ge3-\frac{1}{4}\times2-\frac{1}{12}+\frac{1}{12}=0$.

$\bullet$ $n_2(v)=3$. If $v$ is a $3_{2,1,0}$-vertex, then $v_3$ cannot a $3_{2,1,0}$-vertex or bad $3_{1,1,0}$-vertex by Lemma~\ref{3210} and ~\ref{3200}. So $v$ gives nothing to $v_3$. By Lemma~\ref{3111} $u_2$ cannot be a $3_{1,1,1}$-vertex. So $v$ gives nothing to $u_2$. By (R1)(R2) $v$ gives $\frac{1}{4}$ to each of $v_1,u_1,v_2$ and gets $\frac{1}{4}$ from $v_3$. So $\mu^*(v)\ge3-{\frac{1}{4}}\times3+\frac{1}{4}=\frac{5}{2}$. If $v$ is a $3_{1,1,1}$-vertex, then by Lemma~\ref{3111} each of $u_1,u_2,u_3$ cannot be a $3_{1,1,1}$-vertex. So $v$ gives $\frac{1}{4}$ to each of $v_1,v_2,v_3$ by (R1) and gets $\frac{1}{12}$ from each of $u_1,u_2,u_3$ by (R4). So $\mu^*(v)\ge3-{\frac{1}{4}}\times3+\frac{1}{12}\times3=\frac{5}{2}$.

{\bf Case 2.} $d(v)=4$. By Lemma~\ref{nearby2} $n_2(v)\le5$.

$\bullet$ $n_2(v)\le3$. Then $a_1\le2, a_2\le1,a_3\le1,a_4=0$. So by (R1)-(R4)$v$ gives at most $\frac{1}{4}\times2$ to $v_1,u_1$, at most $\frac{1}{4}+\frac{1}{12}=\frac{1}{3}$ to $v_i,u_i$ for each $i\in\{2,3,4\}$. So $\mu^*(v)\ge4-\frac{1}{4}\times2-\frac{1}{3}\times3=\frac{5}{2}$.

$\bullet$ $n_2(v)=4$. Then $v$ is either a $4_{2,2,0,0}$-vertex or a $4_{2,1,1,0}$-vertex or a $4_{1,1,1,1}$-vertex. By (R1) $v$ give $\frac{1}{2}$ to each nearby $2$-vertex. In the first case, $v$ gives at most $\frac{1}{4}$ to each of $v_3,v_4$ by (R2)(R3). In the second  case, $v$ gives at most $\frac{1}{12}$ to each of $v_2,v_3$ and at most $\frac{1}{4}$ to $v_4$ by (R2)(R3). In the last case, $v$ gives at most $\frac{1}{12}$ to each $v_i$ for $i\in[4]$ by (R4). So $\mu^*(v)\ge4-\frac{1}{4}\times4-\max\{\frac{1}{4}\times2,\frac{1}{4}+\frac{1}{12}\times2,\frac{1}{12}\times4\}=\frac{5}{2}$.

$\bullet$ $n_2(v)=5$. Then $v$ is either a $4_{2,2,1,0}$-vertex or a $4_{2,1,1,1}$-vertex. By (R1) $v$ give $\frac{1}{2}$ to each nearby $2$-vertex. In the former case, $v_4$ cannot be a $3_{2,1,0}$-vertex by Lemma~\ref{42210}. So by (R3)(R4) $v$ give at most $\frac{1}{12}$ to each of $u_3,v_4$. So $\mu^*(v)\ge4-\frac{1}{4}\times5-\frac{1}{12}\times2=\frac{31}{12}>\frac{5}{2}$.
In the latter case, by (R4) $v$ gives at most $\frac{1}{12}$ to each of $u_2,u_3,u_4$. So $\mu^*(v)\ge4-\frac{1}{4}\times5-\frac{1}{12}\times3=\frac{5}{2}$.

\end{document}